\documentclass{article}
\usepackage{amsfonts, amsthm, amsmath, mathtools, amssymb, graphicx, comment, xcolor, enumitem, mathabx,xfrac, mathrsfs, indentfirst, esint, bm, slashed}
\usepackage{latexsym}
\usepackage{amssymb}
\usepackage{amsxtra}
\usepackage{geometry}
\usepackage{setspace}

\newtheorem{lemma}{Lemma}[section]
\newtheorem{thm}{Theorem}[section]

\newtheorem{defn}{Definition}[section]

\newtheorem{rmk}{Remark}[section]
\newcommand{\ZZ}{\mathbb{Z}}

\newcommand{\R}{\mathbb{R}}
\newcommand{\C}{\mathbb{C}}

\newcommand{\nab}{\nabla}

\newcommand{\W}{\Omega}
\newcommand{\w}{\omega}
\newcommand{\be}{\beta}

\newcommand{\FS}{\mathcal{S}}

\newcommand{\Comp}{\mathfrak{C}}

\newcommand{\pex}{p_{\mathrm{ext}}}
\newcommand{\pext}[1]{p_{\mathrm{ext},#1}}

\newcommand{\vel}{\mathbf{v}}

\newcommand{\x}{\mathbf{x}}

\newcommand{\E}{\mathcal{E}}
\newcommand{\al}{\alpha}

\newcommand{\y}{\gamma}

\newcommand{\pr}{\prime}

\newcommand{\un}{\mathbf{\hat{n}}}

\newcommand{\vphi}{\varphi}

\newcommand{\half}{\sfrac{1}{2}}

\newcommand{\D}{\Delta}
\newcommand{\de}{\delta}
\newcommand{\chiw}{\chi_{\w}}

\newcommand{\ee}{\varepsilon}
\newcommand{\g}{\mathbf{g}}

\newcommand{\sig}{\sigma}

\newcommand{\tta}{\theta}

\newcommand{\p}{\partial}

\newcommand{\abs}[1]{\left\lvert#1\right\rvert}

\newcommand{\set}[1]{\left\{#1\right\}}

\newcommand{\T}{\mathbb{T}}
\newcommand{\N}{\mathbb{N}}

\newcommand{\norm}[1]{\left\|#1\right\|}
\newcommand{\comm}[2]{\left[#1,#2\right]}

\newcommand{\Dist}{\mathcal{D}^\pr}

\newcommand{\hnorm}[2]{\norm{#1}_{\dot{H}^{#2}}}

\def\avint{\mathop{\mathchoice{\,\rlap{-}\!\!\int}
                              {\rlap{\raise.15em{\scriptstyle -}}\kern-.2em\int}
                              {\rlap{\raise.09em{\scriptscriptstyle -}}\!\int}
                              {\rlap{-}\!\int}}\nolimits}
                              
\newcommand{\Norm}[2]{\norm{#1}_{H^{#2}}}
\newcommand{\vbar}{v^*}

\newcommand{\Lip}{\mathrm{Lip}}
\newcommand{\pnorm}[2]{\norm{#1}_\Lp{#2}}
\newcommand{\Ec}{\E_{\mathrm{cap}}}
\newcommand{\Eg}{\E_{\mathrm{grav}}}
\newcommand{\Eca}[1]{\E_{\mathrm{cap},#1}}
\newcommand{\Egr}[1]{\E_{\mathrm{grav},#1}}

\DeclareMathOperator{\Rea}{\mathfrak{Re}}

\DeclareMathOperator{\FT}{\mathcal{F}}

\DeclareMathOperator{\bigo}{\mathcal{O}}

\DeclareMathOperator{\id}{\mathrm{id}}

\DeclareMathOperator{\diver}{\mathrm{div}}
\DeclareMathOperator{\curl}{\mathrm{curl}}
\DeclareMathOperator{\supp}{\mathrm{supp}}
\DeclareMathOperator{\Op}{\mathrm{Op}}

\newcommand{\Lp}[1]{{L^{#1}}}

\newcommand{\Int}{\int_0^{2\pi}}

\let\originalleft\left
\let\originalright\right
\renewcommand{\left}{\mathopen{}\mathclose\bgroup\originalleft}
\renewcommand{\right}{\aftergroup\egroup\originalright}

\numberwithin{equation}{section}

\title{A Toy Model for Damped Water Waves}
\author{Gary Moon}
\date{}

\begin{document}
\maketitle

\begin{abstract}
We consider a toy model for a damped water waves system in a domain $\W_t \subset \T \times \R$. The toy model is based on the paradifferential water waves equation derived in the work of Alazard-Burq-Zuily \cite{ABZ1, ABZ3}. The form of damping we utilize is a modified sponge layer proposed for the $3d$ water waves system by Clamond, et.\@ al.\@ in \cite{ClaEtAl}. We show that, in the case of small Cauchy data, solutions to the toy model exhibit a quadratic lifespan. This is done via proving energy estimates with the energy being constructed from appropriately chosen vector fields.
\end{abstract}

\section{Introduction}

\let\thefootnote\relax\footnote{The author was partially supported by NSF CAREER Grant DMS-1352353 and NSF Applied Math Grant DMS-1909035.}

We study a toy model for damped (irrotational) water waves in two spatial dimensions. Recall that the irrotational water waves system is given by the incompressible, irrotational Euler equations coupled with a kinematic and a dynamic boundary condition on the free surface (i\@.e\@., the free-surface Euler equations). We will work on the torus $\T \coloneqq \R/2\pi\ZZ$, which is equivalent to imposing periodic (horizontal) boundary conditions, but one can also choose to work on the real line which necessitates imposing appropriate decay conditions at infinity. Letting $\W_t$ denote the fluid domain and $\FS_t$ the air-water interface, we have
\begin{equation}
\label{eqn:Free-SurfaceEuler}
\begin{cases} \p_t\vel + (\vel\cdot\nab)\vel = -\nab\frac{p}{\rho_0} + \g & \text{in } \W_t\\ \diver \vel = 0 & \text{in } \W_t\\ \curl \vel = 0 & \text{in } \W_t\\ \p_t + \vel\cdot\nab \text{ is tangent to } \bigcup_t \FS_t \times \set{t} \subset \T_x \times \R_y \times \R_t\\ p = -\tau H(\eta) & \text{on } \FS_t \end{cases}.
\end{equation}
In \eqref{eqn:Free-SurfaceEuler}, $\vel$ is the fluid velocity, $p$ denotes the pressure, $\rho_0$ is the constant density, $\g \coloneqq (0,-g)$ with $g$ being acceleration due to gravity, $\tau$ is the coefficient of surface tension, $\eta$ is a function describing the location of the free surface, $H(\eta)$ is the mean curvature of the free surface. We shall henceforth, by rescaling, assume unity density $\rho_0 = 1$. We shall take $\W_t \subset \T \times \R$ to be given by
\begin{equation}
\label{eqn:FluidDomain}
\W_t \coloneqq \set{(x,y) \in \T \times \R : -\infty < y < \eta(x,t)},
\end{equation}
so the domain $\W_t$ is of infinite vertical extent. The free surface is, of course, given by
\begin{align}
\FS_t &\coloneqq \mathcal{G}_{\eta,t}, \label{eqn:FreeSurface}
\end{align}
where $\mathcal{G}_{\eta,t} \coloneqq \set{(x,\eta(x,t)): x \in \T}$ denotes the graph of $\eta$ at time $t \geq 0$.

A large part of our motivation for studying damped water waves is due to the importance of damping in the numerical study of water waves. For example, one may want to numerically study the behavior of surface water waves in an effectively unbounded domain (e\@.g\@., waves on the open ocean), but, when running a numerical simulation on a computer, one is necessarily required to truncate the domain. The artificial boundaries introduced by this truncation can cause problems if not handled appropriately. For example, waves can reflect off a component of the boundary and propagate back into the domain, which creates interference.

One effective solution to this problem is to introduce a wave damper near some subset of the boundary to damp the outgoing waves so that they are not reflected back into the wave propagation region. This is, of course, not the only means to solve this problem (e\@.g\@., there are absorbing boundary conditions and perfectly matched layers), however damping can provide a very effective solution that is also quite computationally efficient.

The particular type of damping which we consider here is a form of modified sponge layer proposed by Clamond, et\@. al\@. in \cite{ClaEtAl} for the damping of $3d$ water waves. This damper acts via an external pressure, denoted $\pex$, at the free surface. We shall let $\w \subset [0,2\pi)$ denote the connected interval on which we will damp the fluid, we let $\chi_{\w}$ be a smooth, non-negative cut-off function supported on $\w$ and $\vphi$ denotes the velocity potential. The external pressure is then given by
\begin{equation}
\label{eqn:ClamondDamper}
\pex \coloneqq \p_x^{-1}(\chi_{\w}\p_x\vphi).
\end{equation}
We refer to damping by such $\pex$ as Clamond damping or linear $H^{\half}$ damping. We note that \eqref{eqn:ClamondDamper} is simply the $2d$ analog of the $3d$ damper introduced in \cite{ClaEtAl}:
\begin{equation}
\label{eqn:ClamondDamper3d}
\pext{3d} = \nab^{-1}\cdot(\chi_{3d}\nab\vphi) \ (\text{modulo a Bernoulli constant}).
\end{equation}
We do not explicitly deal with the Bernoulli constant as it is a function of time alone and so will not have any impact on energy estimates. On the other hand, the Bernoulli constant can be quite important computationally. Though in some contexts it can be incorporated into the velocity potential and thus taken to vanish (in fact, this gives a sort of gauge condition enforcing uniqueness of the velocity potential), it cannot generically be set to zero. Generally, the treatment of the Bernoulli constant will depend upon the method one uses to resolve the equations.

It is often preferable not to work directly with the free-surface Euler equations \eqref{eqn:Free-SurfaceEuler}, but to instead reformulate it. In fact, we can reduce \eqref{eqn:Free-SurfaceEuler} to a system on the interface $\FS_t$. We will consider the water waves problem from the point of view of the Zakharov-Craig-Sulem formulation \cite{Z, CS}, however there are many ways that one could reformulate the problem (see Chapter 1 of \cite{Lan1} for a discussion of several different possible formulations of the water waves problem).

Given the assumptions on $\vel$ in \eqref{eqn:Free-SurfaceEuler} (namely, irrotationality), it follows that there exists a scalar potential $\vphi$ such that $\vel = \nab\vphi$. Let $\psi$ denote the trace of the velocity potential $\vphi$ on the interface $\FS_t$ and recall that $\eta$ describes the location of the free surface. Then, $(\eta,\psi)$ solves
\begin{equation}
\label{eqn:ZakharovSystem}
\begin{dcases} \p_t\eta - G(\eta)\psi = 0\\ \p_t\psi + g\eta - \tau H(\eta) + \frac{1}{2}(\p_x\psi)^2 - \frac{1}{2}\frac{(G(\eta)\psi + \p_x\eta\p_x\psi)^2}{1+(\p_x\eta)^2}=-\pex\end{dcases},
\end{equation}
where $\pex$ is given by \eqref{eqn:ClamondDamper}, $G(\eta)$ is the (normalized) Dirichlet-Neumann map given by
\begin{align}
\label{eqn:DNO}
G(\eta)\psi(x,t) &\coloneqq \sqrt{1+(\p_x\eta(x,t))^2}\p_\un\vphi(x,\eta(x,t),t) \nonumber\\
&= \p_y\vphi(x,\eta(x,t),t) - \p_x\eta(x,t)\p_x\vphi(x,\eta(x,t),t)
\end{align}
and we let $\un$ denote the outward unit normal vector field on $\FS_t$. We take $H(\eta)$ to denote the mean curvature of the free surface:
\begin{equation}
\label{eqn:MeanCurvature}
H(\eta) \coloneqq \p_x\left( \frac{\p_x\eta}{\sqrt{1+(\p_x\eta)^2}} \right).
\end{equation}
 
Our toy model will be built from the paradifferential equation for the water waves system. This paradifferential equation originates in the beautiful work of Alazard-Burq-Zuily \cite{ABZ1}, which considers \eqref{eqn:ZakharovSystem} with $\pex \equiv 0$. The paradifferential approach to the study of water waves began with the work of Alazard-M\'{e}tivier on the regularity of three-dimensional diamond waves \cite{AM}. We briefly recall Alazard-Burq-Zuily's result in the context of the $2d$ gravity-capillary water waves system, referring to \cite{ABZ1} for the details. Let $(V,B)$ be the trace of $\nab\vphi$ along the free surface. Then, we have the following paradifferential equation for the water waves system:
\begin{equation}
\label{eqn:ParadifferentialWaterWavesEqn}
\p_t U + T_V\p_xU + iT_\y U = f.
\end{equation}
The unknown $U$ is
\begin{equation}
\label{eqn:ParadifferentialUnknown}
U \coloneqq \Comp\left( S\begin{pmatrix} \eta\\ u \end{pmatrix} \right),
\end{equation}
where $\Comp: (a,b) \mapsto a + ib$, $S$ is a paradifferential symmetrizer and $u$ is the good unknown of Alinhac \cite{Ali1, Ali2}: $u \coloneqq \psi - T_B\eta$. The paradifferential operator $T_\y$ incorporates the principal part of the mean curvature operator $H(\cdot)$ and the Dirichlet-Neumann map $G(\eta)$. The corresponding symbol $\y=\y(x,\xi,t)$ is an elliptic symbol of order $\frac{3}{2}$. The right-hand side $f$ consists of smooth(er) remainder terms, which satisfy good estimates. A similar result can be obtained for the gravity water waves system, however in that case $\y$ is of order $\frac{1}{2}$ \cite{ABZ3}.

\subsection{Some Existing Results for the Water Waves System}

The earliest work on the well-posedness of the water waves system made quite restrictive assumptions on the Cauchy data and the geometry. Many early works considered analytic data and analytic geometry. For example, Kano-Nishida proved local well-posedness of the gravity water waves system under the assumption that the initial data is analytic and the bottom is flat \cite{KN}. Other early works assumed small, perturbative data (in Sobolev spaces) and a small, perturbative bottom in the finite-depth case (e.g., \cite{Nal}). 

Local-in-time well-posedness, in the Sobolev space framework, of the full water waves system (with no perturbative assumptions) goes back to Beyer and G\"{u}nther \cite{BG} (capillary waves) and Wu \cite{Wu2} (gravity waves). The corresponding result for gravity-capillary water waves was obtained by Iguchi \cite{Igu1}. These results have been extended in numerous ways over the past 20+ years - for example, incorporating (non-small) topography, allowing the flow to have non-vanishing vorticity in the bulk of the fluid and allowing for rough Cauchy data \cite{CoSh1, Lan2, ABZ1, HIT1} and one can consult the references therein to delve deeper in any given direction. A good exposition of the water waves problem and related issues, including well-posedness, that is rather current can be found in \cite{Lan1}.

A question related to local well-posedness, which has been the focus of a substantial amount of research and is the subject of this paper, regards the lifespan of solutions to the water waves problem, generally in the small-data setting (i.e., Cauchy data of size $\ee \ll 1$), and the long-time behavior of solutions. The water waves system exhibits a quadratic nonlinearity, which suggests a lifespan at least on the order of $\frac{1}{\ee}$ by the classical theory of quasilinear hyperbolic equations \cite{Kato1, Kato2}.

In \cite{HIT1}, Hunter-Ifrim-Tataru apply their ``modified energy method'' to the gravity water waves system. The main idea behind the method is to use a normal form transformation to construct a modified energy functional which satisfies cubically nonlinear estimates. They thereby show that the $2d$ gravity water waves system, formulated in holomorphic coordinates, is locally-in-time well-posed on cubic (i.e., $\bigo(\frac{1}{\ee^2})$) timescales. The results of \cite{HIT1} have been extended in several directions (e.g., most recently to gravity waves with constant vorticity \cite{IT3}). For some interesting results on long-time existence for periodic gravity-capillary water waves see \cite{IonPus5}.

Other works consider the lifespan of solutions to the water waves system, but measured in terms of various dimensionless constants encoding important properties of the flow as opposed to the size of the Cauchy data. Such work is generally undertaken with an eye toward the rigorous justification of various simplified models for water waves in asymptotic regimes (e.g., KdV, Green-Naghdi and so forth). Some commonly used dimensionless constants include
\begin{equation}
\label{eqn:DimensionlessConsts}
\epsilon \coloneqq \frac{a}{H}, \ \mu \coloneqq \frac{H^2}{L^2}, \ \be \coloneqq \frac{b}{H},
\end{equation}
where $a$ is the order of amplitude of the free surface waves, $H$ is the characteristic water depth, $L$ is the characteristic wavelength in the longitudinal direction and $b$ is the order of amplitude of the bathymetric variations. It is common to call $\epsilon$ the nonlinearity parameter, $\mu$ the shallowness parameter and $\be$ the topography parameter. In \cite{AlvLan1}, it was shown that solutions to the water waves system persist on timescales of order $\bigo(\frac{1}{\be\vee\epsilon})$. The main tools included a detailed study of the Dirichlet-Neumann map and a Nash-Moser iteration scheme. This was improved by M\'{e}sognon-Gireau, who proved in \cite{BMG1} that solutions have a lifespan of order $\bigo(\frac{1}{\epsilon})$. The analysis of \cite{BMG1} is of particular interest to us due to the author's use of the commuting vector field $\epsilon\p_t$.

Intimately related to the lifespan of solutions is the existence of global or almost-global solutions starting from small, localized, smooth Cauchy data. Of course, if one has a global solution to the water waves system, it is natural to ask about the long-time asymptotic behavior of the solution, such as whether the solution scatters to a linear solution as to $\to +\infty$. Global existence results are generally easier to prove in $3d$ than in $2d$ due to the better decay in higher dimensions. Additionally, these results are more easily obtained for the gravity water waves system than for the gravity-capillary or capillary water waves system. Here, the difference is related to the dispersion relation and resonant interactions. In three dimensions, global solutions have been shown to exist for the gravity, capillary and gravity-capillary problems by various authors. For example, in \cite{GMS2}, Germain-Masmoudi-Shatah study the $3d$ capillary water waves system, using their celebrated method of space-time resonances to prove the existence of a global solution and to show that said solution scatters to a solution of the linearized system as $t \to +\infty$. For the $2d$ case, global existence is known for gravity water waves and capillary water waves. For example, Alazard-Delort prove global regularity of gravity water waves and a modified scattering via an approach utilizing paradifferential normal forms and Klainerman vector fields \cite{AD2} On the other hand, to the best of the  author's knowledge, the global regularity problem for the $2d$ gravity-capillary water waves system is still open and the best result is the almost-global well-posedness result of Berti-Delort \cite{BD1}.

\subsection{Some Existing Results on the Damped Water Waves Problem and Stabilization/Control of the Water Waves System}

When we refer to damping water waves, we are referring to the application of a sponge layer or numerical beach; that is, an artificial, dissipative term supported near the boundary that removes energy from the system. However, in the literature, there are other systems which are referred to as models for damped water waves. We will briefly give an overview of some of this material and ultimately discuss how it differs from the damping we consider here.

We begin with the damped Euler equations:
\begin{equation}
\label{eqn:DampedEulerEqn}
\begin{dcases}
\p_t\vel + (\vel\cdot\nab)\vel + a\vel = -\nab \frac{p}{\rho_0} + \mathbf{f}\\
\diver\vel = 0
\end{dcases},
\end{equation}
where $a > 0$ is the damping coefficient and $\mathbf{f}$ represents any body forces acting on the flow. If one wants to study the corresponding free boundary problem, the kinematic and dynamic boundary conditions are the same. To see that $a\vel$ is indeed dissipative, consider the energy of the gravity water waves system corresponding to the damped Euler equations:
\begin{equation}
\label{eqn:DampedEulerGravEnergy}
\frac{d}{dt}\left( \frac{1}{2}\int_{\W_t} \abs{\vel}^2 \ d\x dy + \frac{g}{2}\int_{\FS_t}\abs{\eta}^2 \ d\x \right) = -a\int_{\W_t} \abs{\vel}^2 \ d\x dy,
\end{equation}
where $\x$ denotes the horizontal coordinates and $y$ the vertical coordinates. The gravity water waves system for the damped Euler equations was considered by Lian in \cite{Lian1} and shown to be globally well-posed.

It is well-known that viscosity is physically dissipative. One can readily see this mathematically by comparing the Euler and Navier-Stokes equations. As such, it is reasonable to suppose that somehow adding viscosity into the water waves system would provide a damping effect. One must then decide how to go about incorporating viscous effects. For example, one might decide to just consider the free-surface Navier-Stokes equations. However, one disadvantage of doing this is that one no longer has a velocity potential and this elucidates a substantial obstruction: generally speaking, viscosity is not compatible with potential flow and the existence of a velocity potential is critical for many of the formulations of the water waves problem (this is also the reason why the study of water waves with vorticity has such a different character). Thus, if one wants to retain the many nice features of potential flows while incorporating the effects of viscosity, then one must consider various ways to add ``artificial viscosity''. One well-studied model of this form is the Dias-Dyachenko-Zakharov (DDZ) model, which, for $2d$ gravity waves over a fluid of infinite depth is
\begin{equation}
\label{eqn:DDZ}
\begin{dcases}
\D\vphi = 0 &\text{in } \W_t\\
\p_t\eta = \p_y\vphi + 2\nu\p_x^2\vphi - \p_x\eta\p_x\vphi &\text{on } \FS_t\\
\p_t\vphi = -\frac{1}{2}\abs{\nab\vphi}^2 - 2\nu\p_y^2\vphi - g\eta &\text{on } \FS_t\\
\p_y\vphi \to 0 &\text{as } y \to - \infty
\end{dcases},
\end{equation}
where $\nu$ is the coefficient of viscosity \cite{DDZ1}. The DDZ model is also known to be globally well-posed \cite{NgNi1}.

Both of the above might reasonably be referred to as models for damped water waves. Indeed, solutions to the free-surface damped Euler equations decay to equilibrium almost exponentially \cite{Lian1} and solutions to the DDZ system decay to equilibrium exponentially in time \cite{NgNi1}. Nevertheless, the damping considered in these models is quite different from the damping which we consider here. Recall that, as noted above, we are largely interested in damping that can be utilized in the numerical simulation of water waves. Hence, we would like the waves to be able to propagate freely in the majority of the domain and only be damped near the artificial boundary in order to avoid spurious reflections. In this respect, the above models are not appropriate as the fluid is damped on the whole domain. Of course, that is not to say that these models cannot be adapted to that purpose and, in fact, it is easy to see how one could localize the effect of the damping to a desired region. Indeed, this gives rise to fascinating questions for future research. For example, suppose we take $\nu = \nu(x)$ with $\nu$ localized. How does this damper perform? Does it stabilize the water waves system? If so, how rapidly do solutions converge to equilibrium? As interesting as these questions are, at least to this author, they are ultimately beyond the scope of this work and we will not address them further.

There is a vast literature regarding numerical aspects of the damped water waves problem (e\@.g\@., numerically evaluating the performance of various dampers). For just a few examples, the interested reader may consult \cite{BMO1, Bonn1, ClaEtAl, Clem1, Duc1, IsrOrs1, JenEA1} as well as the references therein. However, the literature regarding the analytic study of the damped water waves problem is quite a bit more limited. We reiterate that, by damped, we mean an artificial dissipative term whose effect is localized (e.g., a sponge layer).

Alazard's papers on the stabilization of the water waves system are an important exception \cite{Ala2, Ala3}. In \cite{Ala2}, the damper 
\begin{equation}
\label{eqn:HamiltonianDamper}
\pext{1} = \chi_1G(\eta)\psi
\end{equation}
is considered, where $\chi_1$ is a cut-off function. This is a commonly used damper and a natural choice upon viewing the water waves system as being governed by a Hamiltonian energy. For, consider the Hamiltonian energy of the water waves system:
\begin{equation}
\label{eqn:HamiltonianEnergy}
\mathcal{H} = \frac{g}{2}\int_0^{2\pi} \eta^2 \ dx + \tau \int_0^{2\pi} \sqrt{1 + \eta_x^2} - 1 \ dx + \frac{1}{2}\int_0^{2\pi} \psi G(\eta)\psi \ dx.
\end{equation}
By using the Hamiltonian equations, one can show that
\begin{equation}
\label{eqn:HamiltonianDamperDissipation}
\frac{d\mathcal{H}}{dt} = \int_0^{2\pi} \frac{\de\mathcal{H}}{\de\eta}\p_t\eta + \frac{\de\mathcal{H}}{\de\psi}\p_t\psi \ dx = -\int_0^{2\pi} \p_t\eta\cdot\pext{1} \ dx = -\int_0^{2\pi} \chi_1(G(\eta)\psi)^2 \ dx \leq 0.
\end{equation}
So, one easily sees that $\pext{1}$ removes energy from the system. The main result of Alazard in \cite{Ala2} is that this damper actually stabilizes the $2d$ gravity-capillary water waves system; in particular, the energy decays to zero exponentially in time.

Alazard obtains an analogous result for the two-dimensional gravity water waves system in \cite{Ala3}. In that case, the external pressure is taken to satisfy
\begin{equation}
\label{eqn:GravDamper}
\pext{2}(x,t) = \p_x^{-1}\left(\chi_2(x)\int_{-h}^{\eta(x,t)} \vphi_x(x,y,t) \ dy\right),
\end{equation}
where $\chi_2$ is again a cut-off function. The damper \eqref{eqn:GravDamper} may appear a bit unnatural, particularly compared to the simple damper \eqref{eqn:HamiltonianDamper}. The main obstruction to simply utilizing the Hamiltonian damper \eqref{eqn:HamiltonianDamper} is in difficulties showing that the Cauchy problem is well-posed, except in the linearization. To see why one might choose a damper of the form \eqref{eqn:GravDamper}, consider the Hamiltonian energy for the gravity water waves system which is just \eqref{eqn:HamiltonianEnergy} with $\tau = 0$. Then, computing the derivative with respect to time as in \eqref{eqn:HamiltonianDamperDissipation}, we have
\begin{equation}
\label{eqn:GravHamiltonianDerivative}
\frac{d\mathcal{H}}{dt} = -\int_0^{2\pi} \pext{2}G(\eta)\psi \ dx.
\end{equation}
We can now apply the divergence theorem to the right-hand side to obtain:
\begin{equation*}
\frac{d\mathcal{H}}{dt} = \iint_{\W_t} \diver_{x,y}(\pext{2}\nab_{x,y}\vphi) \ dydx = -\int_0^{2\pi} (\p_x\pext{2})\overline{V} \ dx,
\end{equation*}
where $\overline{V}(x,t) \coloneqq \int_{-h}^{\eta(x,t)} \vphi_x(x,y,t) \ dy$. We now see that choosing $\pext{2}$ of the form \eqref{eqn:GravDamper} will force the energy to dissipate. See \cite{Ala3} for all of the details of the above. The main result of Alazard in \cite{Ala3} is that such a $\pext{2}$ stabilizes the water waves system with the energy decaying to zero exponentially in time.

The question of stabilizability of the water waves equations belongs to the broader field of control theory for water waves. Within control theory, stabilizability is one of three closely related concepts, the other two being controllability and observability. Controllability is also an important question in the numerical simulation of water waves as it is intimately related to the generation of waves via wave makers. As was the case for the problem of stabilizability and damping, there are relatively few analytical results on the control theory of the full (nonlinear) water waves system.

In \cite{ABH}, Alazard, Baldi and Han-Kwan prove that the periodic $2d$ gravity-capillary water waves system is locally exactly controllable in arbitrarily short time, subject to a smallness constraint. This controllability result was extended to higher dimensions in \cite{Zhu1} under the assumption that the control domain $\w \subset \T^d$ satisfies the geometric control condition (GCC) of Rauch-Taylor. The GCC is a very natural requirement on the control domain and was implicit in the $2d$ control result of \cite{ABH}: any control domain $\w \subset \T$ satisfies the GCC. For more on the GCC, see \cite{RT1,BLR1}. Regarding observability of water waves, Alazard proves the boundary observability of the gravity water waves system in $2d$ and $3d$, where the fluid domain is taken to be a rectangular tank bounded by a flat bottom, vertical walls and a free surface \cite{Ala8}. Boundary observability implies that one can control the energy of the system via measurements at the boundary (i\@.e\@., where the free surface meets the vertical walls).

\subsection{The Toy Model for the Damped Water Waves System}

We want to consider a damped form of \eqref{eqn:ParadifferentialWaterWavesEqn}. This leads us to consider the following toy model for the (two-dimensional) water waves system subject to Clamond damping:
\begin{equation}
\label{eqn:ClamondToy}
\begin{dcases} \p_t U + W(U)\p_xU + iL U + \chi_{\w} U = 0\\ U(t=0) = U_0 \in H^\sig \end{dcases}.
\end{equation}
In the above, the unknown $U: \T \to \C$ and $L$ is defined by
\begin{equation}
\label{eqn:LDef}
L \coloneqq \abs{D}^\al \text{ for } \al \in (0,2].
\end{equation}
Observe that $L$ corresponds to the paradifferential operator $T_\y$ in the full paradifferential model \eqref{eqn:ParadifferentialWaterWavesEqn}. Finally, $\chiw$ corresponds to the cut-off function in \eqref{eqn:ClamondDamper}. We shall discuss $W(U)$ in short order. Notice that $\al = \frac{3}{2}$ in \eqref{eqn:LDef} corresponds to capillary waves and $\al = \frac{1}{2}$ corresponds to gravity waves. For discussion of the Cauchy problem for a similar toy model see \cite{Ala2}.

Notice that $W(U)$ is the toy model counterpart to the paraproduct operator $T_V$ in \eqref{eqn:ParadifferentialWaterWavesEqn}. For some integer $N \gg 0$, $W$ is continuous from $H^{s-N} \to H^s$:
\begin{equation}
\label{eqn:W_Estimate}
\Norm{W(U)}{s} \lesssim \Norm{\langle D \rangle^{-N}U}{s}.
\end{equation}
Further, $W(\cdot)$ is real-valued, satisfies $W(U)=W(U^*)$ and scales linearly in $U$ (i.e., $W(\cdot)$ is homogeneous of degree one: $W(\lambda U) = \lambda W(U)$ for $\lambda \in \R$). We use $(\cdot)^*$ to denote complex conjugation. Lastly, we shall assume that $W$ commutes with differentiation with respect to time:
\begin{equation}
\label{eqn:WTimeDerivative}
\p_t W(U) = W(\p_t U).
\end{equation}
We justify the reasonableness of this assumption by recourse to the properties of the paraproduct operator $T_V$. Recall that if $a=a(x)\in L^\infty$, it holds that $T_a: H^s \to H^s$ for any $s \in \R$ with the attendant estimate
\begin{equation}
\label{eqn:ParaproductSobolevMapping}
\Norm{T_a u}{s} \lesssim \pnorm{a}{\infty}\Norm{u}{s};
\end{equation}
for example, see \cite{Met1} or \cite{T2}. Given that $\p_t(T_V \p_xU) = T_V\p_t\p_xU + T_{\p_t V}\p_xU$, we will have
\begin{equation*}
\Norm{\p_t(T_V \p_xU)}{s} \lesssim \pnorm{V}{\infty}\Norm{\p_t \p_xU}{s} + \pnorm{\p_t V}{\infty}\Norm{\p_xU}{s},
\end{equation*}
provided that $V$ and $\p_t V$ are $L^\infty$. Finally, since $\p_t \p_xU \approx \abs{D}^\al \p_xU$ (at least when $\al > 1$), we see that $T_{\p_t V}\p_xU$ can be considered a lower-order remainder term. If $\al \leq 1$, we have $\p_t\p_x U \approx \p_x(W(U)\p_xU) + i\abs{D}^\al \p_xU$ and this will still represent the highest-order term.

So, the only remaining question regards the reasonableness of assuming that $V$ and $\p_t V$ are $L^\infty$. We first note that we are not considering rough solutions/data and it would be entirely reasonable to assume that $V$ has enough regularity to enforce the desired inclusion via Sobolev embedding. Nevertheless, we note that in the low-regularity context of \cite{ABZ1}, this assumption is verified. In particular, $V$ is $H^{s-1}$ and $\p_tV$ is $H^{s-2}$ for $s > 2 + \frac{d}{2}$. So, we have just enough regularity to ensure that $V, \p_tV \in L^\infty$ via Sobolev embdedding. The details can be found in \cite{ABZ1, ABZ3}.

We shall be working in the small-data setting and so assume that
\begin{equation}
\label{eqn:SmallData}
\Norm{U_0}{\sig} = \ee \ll 1.
\end{equation}
Our aim will be to show that solutions to \eqref{eqn:ClamondToy} exist on timescale $\bigo(\frac{1}{\ee})$. We will achieve this by rescaling $U$: $U(x,t) = \ee v(x,\ee t)$. Then, our objective will be to simply obtain uniform-in-$\ee$ estimates on the solution $v$ to the equation
\begin{equation}
\label{eqn:ClamondToyScale}
\begin{dcases} \p_t v + W_\ee(v)\p_xv + \frac{i}{\ee}L v + \frac{1}{\ee}\chi_{\w} v = 0\\ v(t=0) = v_0 \in H^\sig \end{dcases},
\end{equation}
where $W_\ee(v) \coloneqq \ee^{-1}W(\ee v)$. Notice that, due to the scaling linearity and \eqref{eqn:W_Estimate}, $W_\ee(v)$ satisfies
\begin{equation}
\label{eqn:W_Estimate2}
\pnorm{W_\ee(v)}{\infty} \lesssim 1 \text{ and } \pnorm{\p_xW_\ee(v)}{\infty} \lesssim 1.
\end{equation}

\section{Main Results and Plan of the Paper}

As noted in the introduction, our objective is to show that solutions to \eqref{eqn:ClamondToy} have an $\bigo(\frac{1}{\ee})$ lifespan. Our approach will be to consider the rescaled equation \eqref{eqn:ClamondToyScale} and to show that solutions have an $\bigo(1)$ lifespan. Our first task will be to define an energy $\E$ for solutions of \eqref{eqn:ClamondToyScale}. This will be done using carefully chosen vector fields. However, the value of $\al$ in \eqref{eqn:LDef} plays a key role in defining a suitable energy. In particular, the order of $L$, compared to the order $1$ nonlinearity $W_\ee(v)\p_xv$, will determine whether $L$ is principal or sub-principal, and this fact plays an important role in determining the analysis necessary to attack the problem. So, the marked difference in analysis is between the cases $\al > 1$ and $\al \leq 1$. However, rather than focusing on these more general cases, we will largely focus in on the cases $\al = \frac{3}{2}$ and $\al = \frac{1}{2}$. In either case, we have the following result:

\begin{thm}
\label{MainTheorem}
Let $v$ be a solution of \eqref{eqn:ClamondToyScale}, where $\al = \frac{3}{2}$ or $\al = \frac{1}{2}$, and suppose that $\sigma$ is sufficiently large. If $\E$ is the appropriate energy of a solution to \eqref{eqn:ClamondToyScale}, then we have
\begin{equation}
\label{eqn:EnergyEstimateForm}
\frac{d\E}{dt} \lesssim \E.
\end{equation}

Detailed statements of this result can be found in Theorem \ref{CapEnergyEstimate1} for the case $\al = \frac{3}{2}$ and in Theorem \ref{GravEnergyEstimate} for the case $\al = \frac{1}{2}$.
\end{thm}

\begin{rmk}
\label{EnergyToLifespan}
It is crucial that the energy estimate \eqref{eqn:EnergyEstimateForm} is uniform in $\ee$. With such an estimate in hand, a routine Gr\"{o}nwall argument will yield the desired $\bigo(1)$ lifespan. Of course, we can then deduce that solutions to \eqref{eqn:ClamondToy} persist on an $\bigo(\frac{1}{\ee})$ timescale.
\end{rmk}

\begin{rmk}
\label{Grav-CapRmk}
Of course, our definition of $L$ omits the gravity-capillary case, corresponding to $L \coloneqq \sqrt{\abs{D} + \abs{D}^3}$. Though we do not treat this case here, we believe that similar results would obtain. All of our arguments extend immediately the the gravity-capillary case, except the proof of Lemma \ref{LCommutator}. However, with only fairly minor modifications to the low-frequency analysis, the proof of Lemma \ref{LCommutator} should be able to handle $L = \sqrt{\abs{D} + \abs{D}^3}$ or something more general like $L = \sqrt{\abs{D}^\al + \abs{D}^\be}$. The high-frequency analysis should extend immediately.
\end{rmk}

In Section 3, we prove some commutator estimates which are needed to prove the main energy estimates in Section 4. We additionally include two appendices. Appendix A contains a collection of results which we shall need to frequently employ. On the other hand, Appendix B contains an alternative proof of one of the main results.

\section{Preliminary Commutator Estimates}

In proving the main estimates, we shall encounter a number of commutators involving pseudodifferential operators ($\Psi$DO). Specifically, we will be considering positive integer powers of $P_L = \Op(p_L)$, which is given by
\begin{equation}
\label{eqn:PLDef}
P_L \coloneqq iL + \chiw.
\end{equation}
Recall that $L$ is defined in \eqref{eqn:LDef}. In the sequel, we will primarily be concerned with two types of commutators involving $P_L$. They will be of the form $\comm{P_L^k}{f}$ and $\comm{P_L^k}{\p_x}$ for $k \in \N$. Note that we lightly abuse notation by not distinguishing the notation for a function $f$ and the operator $M_f: u \mapsto fu$. To avoid confusion, note that, when $P$ is an operator and $f$ a function, we define the commutator $\comm{P}{f} \coloneqq \comm{P}{M_f}$; that is,
\begin{equation*}
\comm{P}{f}(u) = P(fu) - fPu.
\end{equation*}

The challenge posed by working with $P_L$ is that its symbol $p_L$ is not smooth at $\xi = 0$, otherwise we could use standard $\Psi$DO commutator estimates. However, as a Fourier multiplier, $iL$ commutes with $\p_x$ and so the commutators of the form $\comm{P_L^k}{\p_x}$ will be rather straightforward to understand. In this case, we have the following result:
\begin{lemma}
\label{PLCommutator1}
Let $k \in \N$. Then, for all $s \geq 0$, we have
\begin{equation}
\label{eqn:PLCommutatorEst1}
\Norm{\comm{P_L^k}{\p_x}(u)}{s} \lesssim \Norm{u}{s+k\al-\al}.
\end{equation}
\end{lemma}
\begin{proof}
We shall proceed by induction on $k$. For $k=1$, it is straightforward to verify \eqref{eqn:PLCommutatorEst1}. Namely, by Lemma \ref{SobolevMultiplication}, we have
\begin{equation}
\label{eqn:PLEst0}
\Norm{\comm{P_L}{\p_x}(u)}{s} = \Norm{(\p_x\chiw)u}{s} \lesssim \Norm{u}{s} \ \ \forall s \geq 0.
\end{equation}

Now, assume that, for some fixed $k \in \N$, we have
\begin{equation}
\label{eqn:PLInduction}
\Norm{\comm{P_L^k}{\p_x}(u)}{s} \lesssim \Norm{u}{s+k\al - \al} \ \ \forall s \geq 0.
\end{equation}
Observe that we can write
\begin{equation}
\label{eqn:PLCommutatorSum}
\comm{P_L^{k+1}}{Q}(u) = P_L\comm{P_L^k}{Q}(u) + \comm{P_L}{Q}(P_L^ku),
\end{equation}
where $Q$ is an arbitrary operator. Fixing $s \geq 0$, we can now estimate each term on the right-hand side of \eqref{eqn:PLCommutatorSum} in the $H^s$ norm. We apply the triangle inequality, Lemma \ref{SobolevMultiplication} and the induction hypothesis \eqref{eqn:PLInduction} to the first term:
\begin{align}
\Norm{P_L\comm{P_L^k}{\p_x}(u)}{s} &\leq \Norm{L\comm{P_L^k}{\p_x}(u)}{s} + \Norm{\chiw\comm{P_L^k}{\p_x}(u)}{s} \nonumber\\
&\lesssim \Norm{\comm{P_L^k}{\p_x}(u)}{s+\al} + \Norm{\comm{P_L^k}{\p_x}(u)}{s} \nonumber\\
&\lesssim \Norm{u}{s+k\al}. \label{eqn:PLEst1.1}
\end{align}
The second term on the right-hand side of \eqref{eqn:PLCommutatorSum} is straightforward to estimate in $H^s$ via Lemma \ref{SobolevMultiplication}:
\begin{equation}
\label{eqn:PLEst1.2}
\Norm{\comm{P_L}{\p_x}(P_L^ku)}{s} = \Norm{(\p_x\chiw)P_L^ku}{s} \lesssim \Norm{P_L^ku}{s} \lesssim \Norm{u}{s+k\al}.
\end{equation}
Using equation \eqref{eqn:PLCommutatorSum} and the triangle inequality, we can deduce from \eqref{eqn:PLEst1.1} and \eqref{eqn:PLEst1.2} that
\begin{equation}
\label{eqn:PLEst1}
\Norm{\comm{P_L^{k+1}}{\p_x}(u)}{s} \lesssim \Norm{u}{s+k\al}.
\end{equation}
This completes the proof.
\end{proof}

The commutators of the form $\comm{P_L^k}{f}$ will require a bit more work. Our objective is to prove the following:

\begin{lemma}
\label{PLCommutator2}
Let $k \in \N$. Then, we have
\begin{equation}
\label{eqn:PLCommutatorEst2}
\Norm{\comm{P_L^k}{f}(u)}{s} \lesssim \Norm{f}{r}\Norm{u}{s+k\al - 1},
\end{equation}
provided $r > \frac{3}{2}$, $s \geq 0$ and $s + k\al \leq r$.
\end{lemma}

Notice that if we could prove \eqref{eqn:PLCommutatorEst2} for $k=1$, then proceeding by induction on $k$ and exploiting \eqref{eqn:PLCommutatorSum} would give the result much like in the proof of Lemma \ref{PLCommutator1}. However, $L$ will not commute with multiplication by a function as it did with differentiation. In fact, we have
\begin{equation*}
\comm{P_L}{f}(u) = i\comm{L}{f}(u).
\end{equation*}
Therefore, in order to get such an argument off the ground, we will need to understand how to estimate such a commutator.

As noted above, a key detail here is that $\ell$ (the symbol of $L$) is not smooth at $\xi = 0$ and so classical $\Psi$DO (or Fourier multiplier) commutator estimates will not apply directly. However, notice that if $\phi = \phi(\xi) \in C_c^\infty(T^*\T)$ with $\phi \equiv 0$ in some neighborhood of $\xi = 0$, then $\phi\ell \in S_{1,0}^{\al}$, where $S_{\rho,\delta}^m$ denotes the standard H\"{o}rmander class of symbols of order $m$ and type $\rho,\delta$. Thus, classical $\Psi$DO commutator estimates ``almost'' apply to $L$ and in fact do apply to $L$ as long as we filter out the low frequencies. Given that we should be able to handle the low frequencies with Bernstein-type inequalities for band-limited functions, this indicates that we should be able to adapt classical $\Psi$DO commutator estimates to handle $L$ and doing so is, in fact, our next objective.

There are a great many results on estimating commutators of the form $\comm{\Op(p)}{f}(u)$ in Sobolev spaces, where $p \in S_{\rho,\delta}^m$ or some other appropriate symbol class. The book \cite{T2} is an excellent resource for microlocal analysis in the context of nonlinear partial differential equations, including such commutator estimates. The result which we will use as a basis for building the needed commutator estimate is the following:

\begin{lemma}
\label{PsiDOCommutator}
Let $m \geq 0$. Further, take $r,s \in \R$ such that $r > \frac{3}{2}$, $s \geq 0$ and $s + m \leq r$. Then, for $p = p(x,\xi) \in \mathcal{B}S_{1,1}^m$, we have
\begin{equation}
\label{eqn:PsiDOCommutatorEst}
\Norm{\comm{\Op(p)}{f}(u)}{s} \lesssim \Norm{f}{r}\Norm{u}{s+m-1}.
\end{equation}
In other words, $\comm{\Op(p)}{f}$ is an operator of order $m-1$.
\end{lemma}
\begin{proof}
See Proposition 4.2 in \cite{T4}.
\end{proof}

The symbol class $\mathcal{B}S_{1,1}^m$ contains those $p(x,\xi) \in S_{1,1}^m$ such that
\begin{equation}
\label{eqn:SpectralCondition}
\exists \y \in (0,1) : \ \supp \hat{p}(\tta,\xi) \subset \set{(\tta,\xi) : \abs{\tta} \leq \y\abs{\xi}}.
\end{equation}
What is important for us about this symbol class is that
\begin{equation}
\label{eqn:SymbolInclusion}
S_{1,0}^m \subset \mathcal{B}S_{1,1}^m + S_{1,0}^{-\infty};
\end{equation}
in fact, we have, for any $0 \leq \de < 1$, $S_{1,\de}^m \subset \mathcal{B}S_{1,1}^m + S_{1,0}^{-\infty}$. In addition, $\mathrm{Op}\mathcal{B}S_{1,1}^m$ contains the paradifferential operators of \cite{Bon1}. A similar estimate that would have suited our purposes is Lemma 3.4 in \cite{Ala6}. We could have also used the estimate (3.6.1) from \cite{T2}, which for $p \in S_{1,0}^m$ would give the same estimate as \eqref{eqn:PsiDOCommutatorEst} after applying Sobolev embedding.

We are now going to use Lemma \ref{PsiDOCommutator} to prove the needed estimate for commutators involving $L$.

\begin{lemma}
\label{LCommutator}
Let $L$ be as in \eqref{eqn:LDef}, $r > \frac{3}{2}$, $s \geq 0$ and $s + \al \leq r$. Then, the result of Lemma \ref{PsiDOCommutator} holds with $p = \ell$. Namely, it holds that
\begin{equation}
\label{eqn:LCommutatorEst}
\Norm{\comm{L}{f}(u)}{s} \lesssim \Norm{f}{r}\Norm{u}{s+\al-1}.
\end{equation}
\end{lemma}
\begin{proof}
As we noted previously, the challenge we must overcome is that $\ell$ is not smooth at $\xi = 0$. To deal with this fact, we decompose $L$ into a component with smooth symbol and a low-frequency component. We will be able to control the (non-smooth) low-frequency factor via Bernstein-type inequalities, while Lemma \ref{PsiDOCommutator} will give us control of the high-frequency, but smooth, factor. 

We will begin by constructing a low-frequency filtering operator of Littlewood-Paley theory (such operators are also called partial sum operators or low-frequency cut-off operators). Let $\phi \in C_c^\infty(T^*\T)$, $0 \leq \phi \leq 1$, be a radial function such that
\begin{equation*}
\phi(\xi) = 1 \text{ for } \abs{\xi} \leq \frac{1}{2} \text{ and } \phi(\xi) = 0 \text{ for } \abs{\xi} \geq 1.
\end{equation*}
We then define $S_0:\Dist \to \Dist$ by
\begin{equation*}
S_0u \coloneqq \phi(D)u,
\end{equation*}
noticing that
\begin{equation*}
\supp \FT(S_0 u) \subset \set{\abs{\xi} \leq 1} \text{ for } u \in \Dist.
\end{equation*}
For more on such operators, see \cite{Met1}.

We can now define the aforementioned decomposition: we write $L = L_1 + L_{>1}$, where $L_1 \coloneqq S_0L$ and $L_{>1} \coloneqq (\id - S_0)L$. Then, simply applying the triangle inequality gives
\begin{equation}
\label{eqn:LCommutatorEst0}
\Norm{\comm{L}{f}(u)}{s} \leq \Norm{\comm{L_1}{f}(u)}{s} + \Norm{\comm{L_{>1}}{f}(u)}{s}.
\end{equation}
We now just have to estimate each term on the right-hand side of \eqref{eqn:LCommutatorEst0}.

We will begin with the low-frequency component. As Fourier multipliers, $S_0$ and $L$ commute, so we can use Lemma \ref{SobolevMultiplication} and a Bernstein-type inequality (e.g., Lemma \ref{Bernstein}) to obtain
\begin{align}
\Norm{L_1(fu)}{s} &= \Norm{LS_0(fu)}{s} \lesssim 2^{\al}\Norm{f}{r}\Norm{u}{s}, \label{eqn:LCommutatorEst1.1}\\
\Norm{fL_1u}{s} &\lesssim \Norm{f}{r}\Norm{LS_0u}{s} \lesssim 2^{\al}\Norm{f}{r}\Norm{u}{s}. \label{eqn:LCommutatorEst1.2}
\end{align}
Of course, \eqref{eqn:LCommutatorEst1.1} and \eqref{eqn:LCommutatorEst1.2} imply that
\begin{equation}
\label{eqn:LCommutatorEst1}
\Norm{\comm{L_1}{f}(u)}{s} \lesssim_\al \Norm{f}{r}\Norm{u}{s}.
\end{equation}

On the other hand, we know from \eqref{eqn:SymbolInclusion} that we can write $L_{>1} = \mathcal{B}L_{>1} + R$, where $\mathcal{B}L_{>1} \in \Op\mathcal{B}S_{1,1}^m$ and $R \in \Op S_{1,0}^{-\infty}$. We can then apply Lemma \ref{PsiDOCommutator} to deduce that
\begin{equation}
\label{eqn:LCommutatorEst2.1}
\Norm{\comm{\mathcal{B}L_{>1}}{f}(u)}{s} \lesssim \Norm{f}{r}\Norm{u}{s+\al-1}.
\end{equation}
On the other hand, since $R$ is a smoothing operator, we have
\begin{equation}
\label{eqn:LCommutatorEst2.2}
\Norm{\comm{R}{f}(u)}{s} \leq \Norm{R(fu)}{s} + \Norm{fRu}{s} \lesssim \Norm{fu}{s} + \Norm{f}{r}\Norm{Ru}{s} \lesssim \Norm{f}{r}\Norm{u}{s}.
\end{equation}
Combining the estimates \eqref{eqn:LCommutatorEst0}, \eqref{eqn:LCommutatorEst1}, \eqref{eqn:LCommutatorEst2.1} and \eqref{eqn:LCommutatorEst2.2} yields \eqref{eqn:LCommutatorEst}.

\end{proof}

We now have all of the tools needed to prove the second main commutator estimate:

\begin{proof}[Proof of Lemma \ref{PLCommutator2}]
To begin, observe that, by Lemma \ref{LCommutator}, we have
\begin{equation}
\label{eqn:PLCommutator2BaseCase}
\Norm{\comm{P_L}{f}(u)}{s} = \Norm{\comm{L}{f}(u)}{s} \lesssim \Norm{f}{r}\Norm{u}{s+\al-1} \ \forall s \geq 0.
\end{equation}
Now, assume that, for some fixed $k \in \N$, we have
\begin{equation}
\label{eqn:PLCommutator2InductionHyp}
\Norm{\comm{P_L^k}{f}(u)}{s} \lesssim \Norm{f}{r}\Norm{u}{s+k\al-1} \ \forall s \geq 0.
\end{equation}
Fixing $s \geq 0$, we can again make use of equation \eqref{eqn:PLCommutatorSum}, which gives
\begin{equation}
\label{eqn:PLCommutator2Est0}
\Norm{\comm{P_L^{k+1}}{f}(u)}{s} \leq \Norm{P_L\comm{P_L^k}{f}(u)}{s} + \Norm{\comm{P_L}{f}(P_L^ku)}{s}.
\end{equation}
Then, Lemma \ref{SobolevMultiplication}, \eqref{eqn:PLCommutator2BaseCase} and \eqref{eqn:PLCommutator2InductionHyp} yield
\begin{align}
\Norm{P_L\comm{P_L^k}{f}(u)}{s} &\lesssim \Norm{\comm{P_L^k}{f}(u)}{s+\al} + \Norm{\comm{P_L^k}{f}(u)}{s} \lesssim \Norm{f}{r}\Norm{u}{s+k\al+\al-1}, \label{eqn:PLCommutator2Est1}\\
\Norm{\comm{P_L}{f}(P_L^ku)}{s} &\lesssim \Norm{f}{r}\Norm{P_L^ku}{s+\al-1} \lesssim \Norm{f}{r}\Norm{u}{s+k\al+\al-1}. \label{eqn:PLCommutator2Est2}.
\end{align}
The desired claim then follows by induction.
\end{proof}

Having the above commutator estimates in hand, we are now prepared to prove the desired energy estimates pursuant to the approach outlined in Section 2.

\section{The Main Energy Estimates}

Here our objective is to prove the energy estimates of Theorem \ref{MainTheorem}. This will, of course, require first defining an appropriate energy for solutions to \eqref{eqn:ClamondToyScale}. However, as we noted previously, the value of $\al$ plays an important role in defining a suitable energy. Nevertheless, there are some relevant results which we can prove for $L$ with any value of $\al \in (0,2]$. In the linear case ($W \equiv 0$), one can show that the solution $v$ actually decays (in norm). On the other hand, in the nonlinear case ($W \not\equiv 0$), we show that $v$ satisfies
\begin{equation*}
\frac{d}{dt}\pnorm{v(t)}{2}^2 \lesssim \pnorm{v(t)}{2}^2.
\end{equation*}

After proving the above results, we break our analysis into two cases: $\al = \frac{3}{2}$ and $\al = \frac{1}{2}$. We first consider $\al = \frac{3}{2}$, defining an appropriate energy $\Ec$ and then proving the desired estimate. Finally, we do the same for $\al = \frac{1}{2}$.

\subsection{Linear Damping}

Here we show that, when $W \equiv 0$, the external pressure $\pex$ damps the energy.
In particular, we show that, for any $k \in \N_0$, the $H^{k\al}$ norm of $v$ is decreasing in time and thus is bounded above by $\Norm{v_0}{k\al}$, provided $\sigma \geq k\al$. Recall that $v_0 \in H^\sigma$ by hypothesis. We begin with the following result:
\begin{lemma}
\label{ClamondToyL2Dec}
If $v$ solves \eqref{eqn:ClamondToyScale} with $W \equiv 0$, then $\norm{v(t)}_\Lp{2}$ is decreasing in time and, in particular,
\begin{equation}
\label{eqn:ClamondToyL2Est}
\norm{v}_{L_t^\infty L_x^2} \leq \norm{v_0}_\Lp{2}.
\end{equation}
This claim is valid for any $\al \in (0,2]$.
\end{lemma}
\begin{proof}
Differentiating $\norm{v(t)}_\Lp{2}^2$ with respect to $t$ and passing the derivative through the integral gives
\begin{equation}
\label{eqn:L2Derivative}
\frac{d}{dt}\norm{v(t)}_\Lp{2}^2 = \Int \vbar \p_t v + v\p_t\vbar \ dx.
\end{equation}
Noting that $\vbar$ also solves \eqref{eqn:ClamondToyScale}, substituting from \eqref{eqn:ClamondToyScale} into \eqref{eqn:L2Derivative} and doing some simple computations yields
\begin{align*}
\frac{d}{dt}\norm{v(t)}_\Lp{2}^2 &= -\Int \vbar\left( \frac{i}{\ee}Lv + \frac{1}{\ee}\chi_{\w} v \right) \ dx - \Int v\left( \frac{i}{\ee}L\vbar + \frac{1}{\ee}\chi_{\w}\vbar \right) \ dx\\
&= -\frac{1}{\ee}\Int \vbar\chi_{\w} v + v\chi_{\w}\vbar \ dx - \frac{i}{\ee}\Int \vbar L v + vL\vbar \ dx 
\end{align*}

We can then slightly rewrite the second integral on the right-hand side above to see that it is purely imaginary:
\begin{equation}
\label{eqn:ClamondToyL2Dec1}
\frac{d}{dt}\norm{v(t)}_\Lp{2}^2 = -\frac{2}{\ee}\left( \Int \chi_{\w}\abs{v}^2 \ dx + i \Int \abs{\sqrt{L}v}^2 \ dx \right).
\end{equation}
But the left-hand side of \eqref{eqn:ClamondToyL2Dec1} is real-valued, so the imaginary part of the right-hand side must vanish. We therefore have
\begin{equation}
\label{eqn:Toy1L2DecEst}
\frac{d}{dt}\norm{v(t)}_\Lp{2}^2 = -\frac{2}{\ee}\Int \chi_{\w}\abs{v}^2 \ dx \leq 0.
\end{equation}
Hence, $\norm{v(t)}_\Lp{2}$ is decreasing in $t$, from which \eqref{eqn:ClamondToyL2Est} immediately follows.
\end{proof}

\begin{lemma}
\label{LinearDamping}
Let $k \in \N_0$ be arbitrary. If $v$ solves \eqref{eqn:ClamondToyScale}, where $W \equiv 0$ and $\sigma \geq k\al$, then $\Norm{v(t)}{k\al}$ is decreasing and is thus controlled by $\Norm{v_0}{k\al}$:
\begin{equation}
\label{eqn:ICControl1}
\norm{v}_{L_t^\infty H_x^{k\al}} \lesssim \Norm{v_0}{k\al}.
\end{equation}
Again, we note that this result holds for any $\al \in (0,2]$.
\end{lemma}
\begin{proof}
Let $Z$ denote the vector field $\ee\p_t$ and consider $Zv$. Observing that $Zv = -iLv - \chi_{\w}v$, we will have $Zv \in L^2$ as long as $Lv \in L^2$ (i.e., $\sig \geq \al$ in equation \eqref{eqn:ClamondToyScale}). In addition, since $Z$ commutes with $\p_t$, $L$ and $\chi_{\w}$, we have
\begin{equation*}
\p_tZv + \frac{i}{\ee}LZv + \frac{1}{\ee}\chi_{\w} Zv = 0.
\end{equation*}
Hence, by Proposition \ref{ClamondToyL2Dec}, $\norm{Zv(t)}_\Lp{2}$ is decreasing; that is, $\norm{(iL + \chi_{\w}) v(t)}_\Lp{2}$ is decreasing and in particular
\begin{equation}
\label{eqn:Toy1L2Dec2}
\norm{(iL + \chi_{\w}) v}_\Lp{2} \lesssim \Norm{v_0}{\al}.
\end{equation}
It then follows that $v \in H^\al$ with 
\begin{equation}
\label{eqn:ClamondToySobolevEst}
\norm{v}_{L_t^\infty H_x^\al} \lesssim \Norm{v_0}{\al}.
\end{equation}

Consider now $Z^2v$. We will have $Z^2v \in L^2$ whenever $\sigma \geq 2\al$. As before, $Z^2v$ solves \eqref{eqn:ClamondToyScale} with $W \equiv 0$ and, again, $\norm{Z^2v(t)}_\Lp{2}$ is decreasing in $t$. But,
\begin{align*}
Z^2 v = (iL + \chi_{\w})^2v = (-L^2 + iL\chi_{\w} + i\chi_{\w} L + \chi_{\w}^2)v.
\end{align*}
In other words, $\norm{(L^2 - i(L\chi_{\w} + \chi_{\w} L) - \chi_{\w}^2)v(t)}_\Lp{2}$ is decreasing and we thus have
\begin{equation}
\label{eqn:ClamondToyL2Dec3}
\norm{(L^2 - i(L\chi_{\w} + \chi_{\w} L) - \chi_{\w}^2)v}_\Lp{2} \lesssim \Norm{v_0}{2\al}.
\end{equation}
Therefore, we deduce from \eqref{eqn:ClamondToyL2Dec3} that $v \in H^{2\al}$ with the estimate
\begin{equation}
\label{eqn:ClamondToySobolevEst2}
\norm{v}_{L_t^\infty H_x^{2\al}} \lesssim \Norm{v_0}{2\al}.
\end{equation}
We continue to iterate this argument and see that, for any $k \in \N_0$, $\norm{Z^kv(t)}_\Lp{2}$ is decreasing, which implies that $\norm{(iL + \chi_{\w})^kv(t)}_\Lp{2}$ is decreasing. We therefore conclude that,
as long as the initial data is sufficiently regular ($\sigma \geq k\al$), $v \in H^{k\al}$ and
\begin{equation}
\label{eqn:ClamondToySobolevEst4}
\norm{v}_{L_t^\infty H^{k\al}_x} \lesssim \Norm{v_0}{k\al}.
\end{equation}
\end{proof}

\subsection{A Nonlinear $L^2$ Estimate}

At this point, we are ready to turn on $W$ and so we shall henceforth remove the assumption that $W\equiv 0$. We shall first obtain an a priori $L^2$ bound and then focus on the higher Sobolev estimates.

\begin{lemma}
\label{ClamondToyL2Control}
If $v$ is a solution of the Clamond toy model \eqref{eqn:ClamondToyScale}, then
\begin{equation*}
\frac{d}{dt}\norm{v(t)}_\Lp{2}^2 \lesssim \norm{v(t)}_\Lp{2}^2.
\end{equation*}
Once more, the above is valid for any $\al \in (0,2]$.
\end{lemma}
\begin{proof}
We begin by differentiating $\pnorm{v(t)}{2}$, passing the derivative through the integral, substituting from \eqref{eqn:ClamondToyScale} and using the fact that $W_\ee(v)=W_\ee(\vbar)$ to obtain
\begin{equation}
\label{eqn:ClamondToyL2Derivative}
\frac{d}{dt}\pnorm{v(t)}{2}^2 = -\Int W_\ee(v)\p_x\abs{v}^2 \ dx - \frac{2i}{\ee}\Int \abs{\sqrt{L}v}^2 \ dx - \frac{2}{\ee}\Int \chi_{\w}\abs{v}^2 \ dx.
\end{equation}
Since the left-hand side of \eqref{eqn:ClamondToyL2Derivative} is real-valued, the imaginary part of the right hand side must vanish and so we will have
\begin{equation}
\label{eqn:ClamondToyL2Derivative2}
\frac{d}{dt}\pnorm{v(t)}{2}^2 = -\Int W_\ee(v)\p_x\abs{v}^2 \ dx - \frac{2}{\ee}\Int \chi_{\w}\abs{v}^2 \ dx.
\end{equation}
We can then integrate by parts in the first term in \eqref{eqn:ClamondToyL2Derivative2} to obtain
\begin{equation*}
\frac{d}{dt}\norm{v(t)}_\Lp{2}^2 = \Int \p_xW_\ee(v)\abs{v}^2 \ dx - \frac{2}{\ee}\Int \chi_{\w}\abs{v}^2 \ dx \leq \Int \p_xW_\ee(v)\abs{v}^2 \ dx.
\end{equation*}
We can then factor out $\pnorm{\p_xW_\ee(v)}{\infty}$ and, by \eqref{eqn:W_Estimate2}, it follows that
\begin{equation}
\label{eqn:ClamondToyL2Control}
\frac{d}{dt}\norm{v(t)}_\Lp{2}^2 \lesssim \norm{v(t)}_\Lp{2}^2.
\end{equation}
\end{proof}

Having obtained the needed $L^2$ estimate, we are now going to define an energy for solutions of \eqref{eqn:ClamondToyScale} in order to obtain the desired Sobolev estimates. At this point, the analysis becomes more sensitive to the value of $\al$, hence we shall stop considering general $\al \in (0,2]$ and break our consideration up into two cases, $\al > 1$ and $\al \leq 1$. In particular, as previously noted, we will focus in on $\al = \frac{3}{2}$, corresponding to capillary waves, and $\al = \frac{1}{2}$, corresponding to gravity waves.

\subsection{The Energy Estimate for Capillary Waves ($\al = \frac{3}{2}$)}

\begin{defn}
\label{CapEnergy}
Let $\al = \frac{3}{2}$ and define the energy for a solution of the Clamond toy model \eqref{eqn:ClamondToyScale} by
\begin{equation}
\label{eqn:CapEnergy}
\Ec(t) \coloneqq \sum_{k=0}^2 \norm{Z^kv(t)}_\Lp{2}^2,
\end{equation}
where $Z$ is a given vector field.
\end{defn}

The vector field we will primarily consider in Definition \ref{CapEnergy} will be $Z = \ee\p_t$. Since the simpler choice of $Z = \ee\p_t$ works in the capillary case, we present this argument in the main body of the paper. However, this vector field will not work in the gravity case and there we utilize $Z = P_L$, where $P_L$ is defined in \eqref{eqn:PLDef}. We include an appendix showing that this choice of $Z$ also works in the capillary case, which gives a unified approach to both problems. Either choice of $Z$ will yield
\begin{equation}
\label{eqn:CapEnergySobolev}
\Ec \sim \Norm{v}{3}^2.
\end{equation}

Now that we have a suitable energy in hand, we can proceed to prove the main energy estimate in the case $\al = \frac{3}{2}$.

\begin{thm}
\label{CapEnergyEstimate1}
If $\al = \frac{3}{2}$, $v$ is a solution of \eqref{eqn:ClamondToyScale} with $\sigma \geq 3$ and $\Ec = \Ec(t)$ is given by \eqref{eqn:CapEnergy} with $Z = \ee\p_t$, then one has
\begin{equation}
\label{eqn:CapEnergyEst1}
\frac{d\Ec}{dt} \lesssim \Ec.
\end{equation}
\end{thm}
\begin{proof}
Write $\Ec(t) = \Eca{0}(t) + \Eca{1}(t) + \Eca{2}(t)$ and notice that, by Lemma \ref{ClamondToyL2Control}, we have
\begin{equation}
\label{eqn:CapE1Est1}
\frac{d\Eca{0}}{dt} \lesssim \Ec.
\end{equation}

Upon taking the derivative of $\Eca{1}$ and substituting from \eqref{eqn:ClamondToyScale} for $v_t$ and $v_t^*$, one sees that
\begin{equation}
\label{eqn:CapE2Deriv1}
\frac{d\Eca{1}}{dt} = -\Int Z(W_\ee(v)\p_x v)Z\vbar + ZvZ(W_\ee(\vbar)\p_x\vbar) \ dx - \frac{2i}{\ee}\Int \abs{\sqrt{L}Zv}^2 \ dx - \frac{2}{\ee} \Int \chi_{\w} \abs{Zv}^2 \ dx.
\end{equation}
Since the left-hand side of \eqref{eqn:CapE2Deriv1} is real-valued, the imaginary part will vanish and, after noting that the third term in \eqref{eqn:CapE2Deriv1} has a good sign, we will have
\begin{equation*}
\frac{d\Eca{1}}{dt} \leq -\Rea\set{\Int Z(W_\ee(v)\p_x v)Z\vbar + ZvZ(W_\ee(\vbar)\p_x\vbar) \ dx}.
\end{equation*}
Expanding in the remaining integral using the Leibniz rule, recalling that $W_\ee(v) = W_\ee(\vbar)$ and integrating by parts yields
\begin{equation}
\label{eqn:CapE2Deriv2}
\frac{d\Eca{1}}{dt} \leq \Int \p_x W_\ee(v)\abs{Zv}^2 \ dx - \Rea\set{\Int W_\ee(Zv)\p_x v \cdot Z\vbar + W_\ee(Z\vbar)\p_x\vbar\cdot Zv \ dx}.
\end{equation}
The first integral in \eqref{eqn:CapE2Deriv2} is easily estimated:
\begin{equation}
\label{eqn:CapE2Control1}
\Int \p_x W_\ee(v)\abs{Zv}^2 \ dx \leq \norm{\p_x W_\ee(v)}_\Lp{\infty}\norm{Zv}_\Lp{2}^2 \lesssim \norm{Zv}_\Lp{2}^2. 
\end{equation}
One can then bound the second integral in \eqref{eqn:CapE2Deriv2} as follows:
\begin{equation}
\label{eqn:CapE2Control2}
-\Rea\set{\Int W_\ee(Zv)\p_x v \cdot Z\vbar + W_\ee(Z\vbar)\p_x\vbar\cdot Zv \ dx}\lesssim \pnorm{W_\ee(Zv)}{\infty}\hnorm{v}{1}\norm{Zv}_\Lp{2}.
\end{equation}
It then follows from \eqref{eqn:CapE2Control1} and \eqref{eqn:CapE2Control2} that
\begin{equation}
\label{eqn:CapE2Est1}
\frac{d\Eca{1}}{dt} \lesssim \norm{Zv}_\Lp{2}^2 + \hnorm{v}{1}\norm{Zv}_\Lp{2} \lesssim \Ec.
\end{equation}

Finally, consider the derivative of $\Eca{2}$ with respect to $t$. We compute the derivative, substitute from \eqref{eqn:ClamondToyScale} and, much as before, we will obtain
\begin{equation}
\label{eqn:CapE3Deriv2}
\frac{d\Eca{2}}{dt} \leq -\Rea\set{\Int Z^2(W_\ee(v)\p_x v)Z^2\vbar + Z^2vZ^2(W_\ee(\vbar)\p_x\vbar) \ dx}.
\end{equation}
Expanding via the Leibniz rule, one obtains
\begin{align}
\label{eqn:CapE3Deriv3}
\frac{d\Eca{2}}{dt} &\leq -\Rea\set{\Int W_\ee(Z^2v)\p_x v \cdot Z^2\vbar + W_\ee(Z^2\vbar)\p_x\vbar\cdot Z^2 v \ dx} \nonumber\\
&\hspace{0.5cm} - 2\Rea\set{\Int W_\ee(Zv)\p_x Zv \cdot Z^2\vbar + W_\ee(Z\vbar)\p_xZ\vbar \cdot Z^2v \ dx} \nonumber\\
&\hspace{0.5cm} - \Int W_\ee(v)\p_x Z^2v \cdot Z^2\vbar + W_\ee(\vbar)\p_x Z^2\vbar \cdot Z^2v \ dx \nonumber\\
&= I + II + III.
\end{align}
We first consider $III$ in \eqref{eqn:CapE3Deriv3} and observe that, since $W_\ee(v) = W_\ee(\vbar)$, we may integrate by parts to see that
\begin{equation}
\label{eqn:CapE3Control1}
III = \Int \p_x W_\ee(v)\abs{Z^2v}^2 \ dx \leq \norm{\p_x W_\ee(v)}_\Lp{\infty}\norm{Z^2v}_\Lp{2}^2.
\end{equation}
On the other hand, an appropriate bound on $I$ is immediate:
\begin{equation}
\label{eqn:CapE3Control2}
I \lesssim \norm{W_\ee(Z^2v)}_\Lp{\infty}\norm{\p_x v}_\Lp{2}\norm{Z^2v}_\Lp{2}.
\end{equation}
Finally, we have
\begin{equation}
\label{eqn:CapE3Control3}
II \lesssim \norm{W_\ee(Zv)}_\Lp{\infty}\norm{\p_x Zv}_\Lp{2}\norm{Z^2v}_\Lp{2}.
\end{equation}
Recall that we can control the $H^r$ norm of $v$ with $\Ec$ for $r \leq 3$. Then, since $\hnorm{Zv}{1} \lesssim \Norm{v}{\sfrac{5}{2}}$, it follows from \eqref{eqn:CapE3Deriv3}, \eqref{eqn:CapE3Control1}, \eqref{eqn:CapE3Control2} and \eqref{eqn:CapE3Control3} that
\begin{equation}
\label{eqn:CapE3Est1}
\frac{d\Eca{2}}{dt} \lesssim \norm{Z^2v}_\Lp{2}^2 + \hnorm{v}{1}\norm{Z^2v}_\Lp{2} + \hnorm{Zv}{1}\norm{Z^2v}_\Lp{2} \lesssim \Ec.
\end{equation}

Upon combining \eqref{eqn:CapE1Est1}, \eqref{eqn:CapE2Est1} and \eqref{eqn:CapE3Est1}, we conclude that
\begin{equation}
\label{eqn:CapEnergyEst1Conclusion}
\frac{d\Ec}{dt} \lesssim \Ec.
\end{equation}
\end{proof}

\subsection{The Energy Estimate for Gravity Waves ($\al = \frac{1}{2}$)}

Here we seek to obtain a result analogous to Theorem \ref{CapEnergyEstimate1} when $\al = \frac{1}{2}$. Again, we shall first need to define a suitable energy, but the situation is complicated by the fact that $L$ is now sub-principal. In particular, the simpler vector field $Z = \ee\p_t$ will no longer suffice and we will need a more carefully chosen $Z$.

\begin{defn}
\label{ClamondToyGravEnergy}
Let $\al = \frac{1}{2}$ and define the energy for a solution $v$ of the Clamond toy model \eqref{eqn:ClamondToyScale} by
\begin{equation}
\label{eqn:ClamondToyGravEnergy}
\Eg(t) \coloneqq \sum_{k=0}^4 \norm{Z^kv(t)}_\Lp{2}^2,
\end{equation}
where $Z = P_L$. Recall that $P_L$ is defined in \eqref{eqn:PLDef}.
\end{defn}

Notice that the definition of the energy in this case requires more copies of the vector field $Z$. This arises from the fact that $L$ is now only of order $\frac{1}{2}$, instead of order $\frac{3}{2}$ in the previous case and so we will need more copies in order to close the estimates. Notice that
\begin{equation}
\label{eqn:ClamondToyGravEnergySobolev}
\Eg \sim \Norm{v}{2}^2.
\end{equation}

The vector field $\ee\p_t$, which we used to define $\Ec$, had the benefit of commuting with $\p_x$, $W_\ee$ and functions of $x$. However, the vector field $Z = P_L$ in Definition \ref{ClamondToyGravEnergy} does not have these nice commutation properties and so obtaining the desired energy estimates will require understanding the effects of commuting powers of $Z$ with $\p_x$ and functions of $x$ (and $t$), such as $W_\ee(v)$. This is precisely the motivation for the results obtained earlier in Section 3.

Defining the energy as in Definition \ref{ClamondToyGravEnergy}, we can prove the following energy estimate.

\begin{thm}
\label{GravEnergyEstimate}
If $\al = \frac{1}{2}$, $v$ is a solution of \eqref{eqn:ClamondToyScale} with $\sigma \geq 2$ and $\Eg = \Eg(t)$ is given by \eqref{eqn:ClamondToyGravEnergy}, then one has
\begin{equation*}
\frac{d\Eg}{dt} \lesssim \Eg.
\end{equation*}
\end{thm}
\begin{proof}
Again, we begin by writing $\Eg(t) = \Egr{0}(t) + \Egr{1}(t) + \Egr{2}(t) + \Egr{3}(t) + \Egr{4}(t)$ and noting that the desired result for $\Egr{0}$ holds by Lemma \ref{ClamondToyL2Control}:
\begin{equation}
\label{eqn:GravEnergyEst0}
\frac{d\Egr{0}}{dt} \lesssim \Eg.
\end{equation}

Now, fix $1 \leq k \leq 4$ and consider $\Egr{k}$. Upon differentiating with respect to $t$, substituting from \eqref{eqn:ClamondToyScale} and rewriting a bit, we obtain
\begin{equation}
\label{eqn:EjDerivative}
\frac{d\Egr{k}}{dt} = -\Int Z^k(W_\ee(v)\p_xv)Z^kv^* + Z^kvZ^k(W_\ee(v^*)\p_xv^*) \ dx - \frac{2i}{\ee}\Int \abs{\sqrt{L}Z^kv}^2 \ dx - \frac{2}{\ee}\Int \chiw\abs{Z^kv}^2 \ dx.
\end{equation}
We know that the left-hand side of \eqref{eqn:EjDerivative} is real-valued, so, once more, the imaginary part of the right-hand side must vanish. Moreover, the third term has a good sign. We therefore deduce that
\begin{equation}
\label{eqn:GravEnergyEst0.1}
\frac{d\Egr{k}}{dt} \leq -\Rea\left\{ \Int Z^k(W_\ee(v)\p_xv)Z^kv^* + Z^kvZ^k(W_\ee(v^*)\p_xv^*) \ dx \right\}.
\end{equation}
By adding and subtracting, we can rewrite \eqref{eqn:GravEnergyEst0.1} as 
\begin{align}
\label{eqn:GravEnergyCjDj}
\frac{d\Egr{k}}{dt} &\leq -\Rea\left\{ \Int [Z^k(W_\ee(v)\p_xv) - W_\ee(v)Z^k\p_xv]Z^kv^* + [Z^k(W_\ee(v^*)\p_xv^*) - W_\ee(v^*)Z^k\p_xv^*]Z^kv \ dx \right\} \nonumber\\
&\hspace{0.5cm} - \Rea\left\{ \Int W_\ee(v)Z^k\p_xvZ^kv^* + W_\ee(v^*)Z^k\p_xv^*Z^kv \ dx \right\} \nonumber\\
&= C_k + D_k.
\end{align}

For the commutator term $C_k$, we first apply the Cauchy-Schwartz inequality, which yields
\begin{equation*}
C_k \lesssim \pnorm{\comm{Z^k}{W_\ee(v)}(\p_xv)}{2}\pnorm{Z^kv}{2}.
\end{equation*}
We can now invoke Lemma \ref{PLCommutator2} to finish off the estimate. We then get
\begin{equation}
\label{eqn:GravEnergyEst0.2}
C_k \lesssim \Norm{W_\ee(v)}{r}\Norm{\p_xv}{\sfrac{k}{2}-1}\pnorm{Z^kv}{2} \lesssim \Norm{\p_xv}{\sfrac{k}{2}-1}\pnorm{Z^kv}{2},
\end{equation}
where $r > \frac{3}{2}$ and $r \geq \frac{k}{2}$. More specifically, for $1 \leq k \leq 3$, we can use $r = \frac{3}{2}+$ and, for $k=4$, we can use $r=2$. At worst, since $k \leq 4$, the right-hand side of \eqref{eqn:GravEnergyEst0.2} involves $\Norm{\p_xv}{1} \leq \Norm{v}{2}$. We thus obtain
\begin{equation}
\label{eqn:GravEnergyEst1}
C_k \lesssim \Norm{\p_xv}{\sfrac{k}{2}-1}\pnorm{Z^kv}{2} \lesssim \Eg.
\end{equation}

We now move on to consider $D_k$. Here, we make use of the fact that $W_\ee(v) = W_\ee(v^*)$ and commute $Z^k$ with $\p_x$, which costs us a derivative commutator:
\begin{equation}
\label{eqn:DjCommute}
D_k = -\Int W_\ee(v)\p_x\abs{Z^kv}^2 \ dx - \Rea\left\{ \Int W_\ee(v)\comm{Z^k}{\p_x}(v)Z^kv^* + W_\ee(v^*)\comm{Z^k}{\p_x}(v^*)Z^kv \ dx \right\}.
\end{equation}
We now integrate by parts in the first term and apply the Cauchy-Schwartz inequality to both terms:
\begin{equation}
\label{eqn:GravEnergyEst1.1}
D_k \lesssim \pnorm{\p_xW_\ee(v)}{\infty}\pnorm{Z^kv}{2}^2 + \pnorm{W_\ee(v)}{\infty}\pnorm{\comm{Z^k}{\p_x}(v)}{2}\pnorm{Z^kv}{2}.
\end{equation}
We now see that we can apply Lemma \ref{PLCommutator1}, which gives us
\begin{equation}
\label{eqn:GravEnergyEst1.2}
D_k \lesssim \pnorm{Z^kv}{2}^2 + \Norm{v}{\sfrac{k}{2}-\sfrac{1}{2}}\pnorm{Z^kv}{2}.
\end{equation}
Notice that the Sobolev norm of $v$ is of order at most $\frac{3}{2}$ since $k \leq 4$. As such, we do not have any trouble closing the estimate:
\begin{equation}
\label{eqn:GravEnergyEst2}
D_k \lesssim \pnorm{Z^kv}{2}^2 + \Norm{v}{\sfrac{k}{2}-\sfrac{1}{2}}\pnorm{Z^kv}{2} \lesssim \Eg.
\end{equation}
Upon combining \eqref{eqn:GravEnergyCjDj}, \eqref{eqn:GravEnergyEst1} and \eqref{eqn:GravEnergyEst2}, we have 
\begin{equation}
\label{eqn:GravEnergyEstj}
\frac{d\Egr{k}}{dt} \lesssim \Eg.
\end{equation}

Finally, the estimates \eqref{eqn:GravEnergyEst0} and \eqref{eqn:GravEnergyEstj} give us the desired control of the time derivative of the energy:
\begin{equation}
\label{eqn:GravEnergyEstFinal}
\frac{d\Eg}{dt} \lesssim \Eg.
\end{equation}

\end{proof}

\appendix

\section{Some Useful Results}

In this section, we will state some results which will be used frequently in later sections. We shall let $\T = \R/2\pi\ZZ$ denote the one-dimensional torus (the circle) and similarly $\T^d = \prod_{j=1}^d \T$ denotes the $d$-dimensional torus. We begin with the following Sobolev embedding result:
\begin{lemma}
\label{SobolevEmbedding}
If $r > \frac{d}{2}$, then $H^r(\T^d) \hookrightarrow L^\infty(\T^d)$ with the estimate
\begin{equation}
\label{eqn:SobolevEmbedding1}
\norm{f}_\Lp{\infty} \lesssim \Norm{f}{r}.
\end{equation}
Further, if $r > 1 + \frac{d}{2}$, then $H^r(\T^d) \hookrightarrow \mathrm{Lip}(\T^d)$ with
\begin{equation}
\label{eqn:SobolevEmbedding2}
\norm{f}_{\Lip} \lesssim \Norm{f}{r}.
\end{equation}
\end{lemma}

We shall frequently need Sobolev space estimates for the product of two functions. Of course, we have the well-known Sobolev algebra property:
\begin{lemma}
\label{SobolevAlgebra}
For $r > \frac{d}{2}$, $H^r(\T^d)$ is a Banach algebra and, for $f,g\in H^r(\T^d)$, we have
\begin{equation}
\label{eqn:SobolevAlgebra}
\Norm{fg}{r} \lesssim \Norm{f}{r}\Norm{g}{r}.
\end{equation}
\end{lemma}

While incredibly useful, the Sobolev algebra property will not suffice for our purposes. For we shall often want to either estimate products in a lower-regularity space or estimate products with rougher functions. To this end, we have the following result:
\begin{lemma}
\label{SobolevMultiplication}
Suppose that $f \in H^r(\T^d)$ and $g \in H^t(\T^d)$ with $r + t > 0$. Then, for all $s$ satisfying $s \leq \min(r,t)$ and $s < r + t - \frac{d}{2}$, we have $fg \in H^s$ with the following estimate:
\begin{equation}
\label{eqn:SobolevProductEst}
\Norm{fg}{s} \lesssim \Norm{f}{r}\Norm{g}{t}.
\end{equation}
\end{lemma}
\begin{proof}
For proving results of this type, Bony's paradifferential analysis is an incredibly powerful tool. Namely, we can apply the paraproduct decomposition of \cite{Bon1}:
\begin{equation}
\label{eqn:ParaproductDecomp}
fg = T_fg + T_gf + R(f,g),
\end{equation}
See Theorem C.10 of \cite{B-GS} for details.
\end{proof}

We shall also make use of a Bernstein-type inequality regarding the effect of differentiation on band-limited functions.
\begin{lemma}
\label{Bernstein}
Define $A \coloneqq \set{\frac{1}{2} \leq \abs{\xi} \leq 2}$ and $B = \set{\abs{\xi} \leq 1}$. Take $\al \in \N_0^d$, $p \in \N_0$ and $u \in L^2(\R^d)$. Then, it holds that
\begin{align}
\supp \hat{u} \subset 2^p B \implies \pnorm{\p^\al u}{2} \lesssim 2^{p\abs{\al}}\pnorm{u}{2}, \label{eqn:Bernstein1}\\
\supp \hat{u} \subset 2^p A \implies \pnorm{\p^\al u}{2} \sim 2^{p\abs{\al}}\pnorm{u}{2}. \label{eqn:Bernstein2}
\end{align}
\end{lemma}
\begin{proof}
See Lemma 2.1 in \cite{BCD}.
\end{proof}

\section{An Alternative Proof of Theorem \ref{CapEnergyEstimate1}}

Theorem \ref{CapEnergyEstimate1} is sufficient for the purpose of obtaining $\bigo(1)$ existence time for solutions of \eqref{eqn:ClamondToyScale}. However, in the $\al = \frac{1}{2}$ case, we were no longer able to use the vector field $\ee\p_t$, instead we utilized $iL + \chi_{\w}$. Here, our goal is to show that one can obtain the result of Theorem \ref{CapEnergyEstimate1} using the vector field $iL + \chi_{\w}$. So, both results can be obtained using a somewhat unified approach.

\begin{thm}
\label{ClamondToyCapEnergyEstimate2}
If $\al = \frac{3}{2}$, $v$ is a solution of \eqref{eqn:ClamondToyScale} with $\sigma \geq 3$ and $\Ec = \Ec(t)$ is given by Definition \ref{CapEnergy} with $Z = iL + \chi_{\w}$, then one has
\begin{equation*}
\frac{d\Ec}{dt} \lesssim \Ec.
\end{equation*}
\end{thm}
\begin{proof}
As in the proof of Theorem \ref{CapEnergyEstimate1}, write $\Ec(t) = \Eca{0}(t) + \Eca{1}(t) + \Eca{2}(t)$ and notice that, by Lemma \ref{ClamondToyL2Control}, we have
\begin{equation}
\label{eqn:AltGrav-CapEnergyEst1}
\frac{d\Eca{0}}{dt} \lesssim \Ec.
\end{equation}

We thus begin in earnest by considering the time derivative of $\Eca{k}$ for $k=1,2$:
\begin{equation*}
\frac{d\Eca{k}}{dt} = -\Int Z^k(W_\ee(v)\p_x v)Z^k\vbar + Z^kvZ^k(W_\ee(\vbar)\p_x\vbar) \ dx - \frac{2i}{\ee} \Int \abs{\sqrt{L}Z^kv}^2 \ dx - \frac{2}{\ee} \Int \chi_{\w}\abs{Z^kv}^2 \ dx.
\end{equation*}
As we've seen many times already, we can reduce this to
\begin{equation}
\label{eqn:ClamondToyGrav-CapE2TimeDeriv2}
\frac{d\Eca{k}}{dt} \leq -\Rea\set{\Int Z^k(W_\ee(v)\p_x v)Z^k\vbar + Z^kvZ^k(W_\ee(\vbar)\p_x\vbar) \ dx}.
\end{equation}

We now rewrite the integral on the right-hand side of \eqref{eqn:ClamondToyGrav-CapE2TimeDeriv2} by adding/subtracting and making note of the fact that $W_\ee(v) = W_\ee(\vbar)$:
\begin{align}
\label{eqn:ClamondToyGrav-CapE2Decomp}
\frac{d\Eca{k}}{dt} &\leq -\Rea\left\{\Int [Z^k(W_\ee(v)\p_x v) - W_\ee(v)Z^k\p_x v]Z^k\vbar + [Z^k(W_\ee(\vbar)\p_x\vbar) - W_\ee(\vbar)Z^k\p_x\vbar]Z^kv \ dx\right\} \nonumber\\
&\hspace{0.5cm} - \Rea\left\{\Int W_\ee(v)(Z^k\p_x vZ^k\vbar + Z^k\p_x\vbar Z^kv) \ dx \right\}\nonumber\\
&= C_k + D_k.
\end{align}
We begin by considering the commutator term $C_k$ for which we plainly have
\begin{equation}
\label{eqn:ClamondToyGrav-CapE21Est1}
C_k \lesssim \norm{\comm{Z^k}{W_\ee(v)}(\p_xv)}_\Lp{2}\norm{Z^kv}_\Lp{2}.
\end{equation}
We can apply Lemma \ref{PLCommutator2} to estimate the right-hand side of \eqref{eqn:ClamondToyGrav-CapE21Est1}:
\begin{equation}
\label{eqn:ClamondToyGrav-CapE21Est}
C_k \lesssim \Norm{W_\ee(v)}{r}\Norm{\p_xv}{\sfrac{3k}{2}-1}\norm{Z^kv}_\Lp{2}.
\end{equation}
We will either take $r=\frac{3}{2}+$ for $k=1$ or $r=3$ for $k=2$. In addition, $\frac{3k}{2}-1 \leq 2$ and so the highest Sobolev norm of $v$ appearing is $\Norm{\p_xv}{2} \leq \Norm{v}{3}$. Thus, the energy estimate closes and we have
\begin{equation}
\label{eqn:CapCjEst}
C_k \lesssim \Norm{\p_xv}{\sfrac{3k}{2}-1}\norm{Z^kv}_\Lp{2} \lesssim \Ec.
\end{equation}

We rewrite $D_k$ by commuting $Z$ with $\p_x$ and integrating by parts, which yields
\begin{equation}
\label{eqn:ClamondToyGrav-CapE22}
D_k = \Int \p_xW_\ee(v)\cdot\abs{Z^kv}^2 \ dx - \Rea\left\{\Int W_\ee(v)\comm{Z^k}{\p_x}(v)Z^k\vbar + W_\ee(v^*)\comm{Z^k}{\p_x}(\vbar)Z^kv \ dx\right\}.
\end{equation}
Hence, we have the estimate
\begin{equation}
\label{eqn:ClamondToyGrav-CapE22Est1}
D_k \lesssim \norm{\p_x W_\ee(v)}_\Lp{\infty}\norm{Z^kv}_\Lp{2}^2 + \norm{W_\ee(v)}_\Lp{\infty}\norm{\comm{Z^k}{\p_x}(v)}_\Lp{2}\norm{Z^kv}_\Lp{2}.
\end{equation}
We apply the derivative commutator estimate of Lemma \ref{PLCommutator1} with $\al = \frac{3}{2}$ to bound the commutator term. This gives control via the energy:
\begin{equation}
\label{eqn:ClamondToyGrav-CapE22Est}
D_k \lesssim \norm{Z^kv}_\Lp{2}^2 + \Norm{v}{\sfrac{3k}{2}-\sfrac{3}{2}}\norm{Z^kv}_\Lp{2} \lesssim \Ec.
\end{equation}
We are able to close the above estimate as $\frac{3k}{2} - \frac{3}{2} \leq \frac{3}{2} < 2$ given that $k=2$ is the worst-case scenario. Ergo, upon combining \eqref{eqn:ClamondToyGrav-CapE21Est} and \eqref{eqn:ClamondToyGrav-CapE22Est}, we obtain control of $\frac{d\Eca{k}}{dt}$:
\begin{equation}
\label{eqn:ClamondToyGrav-CapE2Est2}
\frac{d\Eca{k}}{dt} \lesssim \Ec.
\end{equation}
This gives us the desired estimate:
\begin{equation}
\label{eqn:ClamondToyGrav-CapEnergyEstimate}
\frac{d\Ec}{dt} \lesssim \Ec.
\end{equation}

\end{proof}


\begin{thebibliography}{1}





\bibitem{Ala6} T. Alazard, Low Mach number limit of the full Navier-Stokes equations, Arch. Rational Mech. Anal. \textbf{180} (2006) 1-73.

\bibitem{Ala2} T. Alazard, Stabilization of the water-wave equations with surface tension, Ann. PDE \textbf{3}:17 (2017).

\bibitem{Ala3} T. Alazard, Stabilization of gravity water waves, J. Math. Pures Appl. \textbf{114} (2018) 51-84.

\bibitem{Ala8} T. Alazard, Boundary observability of gravity water waves, Ann. I. H. Poincar\'{e} \textbf{35} (2018) 751-779.

\bibitem{ABH} T. Alazard, P. Baldi and D. Han-Kwan, Control of water waves, J. Eur. Math. Soc. \textbf{20} (2018) 657-745.

\bibitem{ABZ1} T. Alazard, N. Burq and C. Zuily, On the water-waves equations with surface tension, Duke Math. J. \textbf{158} (2011) 413-499.


\bibitem{ABZ3} T. Alazard, N. Burq and C. Zuily, On the Cauchy problem for gravity water waves, Invent. Math. \textbf{198} (2014) 71-163.



\bibitem{AD2} T. Alazard and J.-M. Delort, Global solutions and asymptotic behavior for two dimensional gravity water waves, Ann. Sci. \'{E}c. Norm. Sup. (4) \textbf{48} (2015) 1149-1238.

\bibitem{AM} T. Alazard and G. M\'{e}tivier, Paralinearization of the Dirichlet-to-Neumann operator, and regularity of three-dimensional water waves, Comm. Partial Differential Equations \textbf{34} (2009) 1632-1704.

\bibitem{Ali1} S. Alinhac, Paracomposition et op\'{e}rateurs paradiff\'{e}rentiels, Comm. Partial Differential Equations \textbf{11} (1986) 87-121.

\bibitem{Ali2} S. Alinhac, Existence d'ondes de rar\'{e}faction pour des syst\`{e}mes quasilin\'{e}aires hyperboliques multidimensionnels, Comm. Partial Differential Equations, \textbf{14} (1989) 173-230.


\bibitem{AlvLan1} B. Alvarez-Samaniego and D. Lannes, Large time existence for 3D water-waves and asymptotics, Invent. Math. \textbf{171} (2008) 485-541.










\bibitem{BCD} H. Bahouri, J.-Y. Chemin and R. Danchin, {\em Fourier Analysis and Nonlinear Partial Differential Equations,} Grundlehren der mathematischen Wissenschaften 343 (Springer, Heidelberg, 2011).

\bibitem{BMO1} G. Baker, D. Meiron and S. Orszag, Generalized vortex methods for free surface flow problems II: Radiating waves, J. Sci. Comp. \textbf{4} (1989) 237-259.

\bibitem{BLR1} C. Bardos, G. Lebeau and J. Rauch, Sharp sufficient conditions for the observation, control, and stabilization of waves from the boundary, SIAM J. Control Optim. \textbf{30} (1992) 1024-1065.

\bibitem{B-GS} S. Benzoni-Gavage and D. Serre, {\em Multi-dimensional Hyperbolic Differential Equations: First-order Systems and Applications,} (Oxford Univ. Press, Oxford, 2007).

\bibitem{BD1} M. Berti and J.-M. Delort, {\em Almost Global Solutions of Capillary-Gravity Water Waves Equations on the Circle,} Lecture Notes of the Unione Matematica Italiana 24 (Springer, Cham, 2018).




\bibitem{BG} K. Beyer and M. G\"{u}nther, On the Cauchy problem for a capillary drop. Part I: Irrotational motion, Math. Meth. Appl. Sci. \textbf{21} (1998) 1149-1183.

\bibitem{Bonn1} F. Bonnefoy, {\em Mod\'{e}lisation Exp\'{e}rimentale et Num\'{e}rique des \'{E}tats de Mer Complexes,} Ph.D. thesis, Universit\'{e} de Nantes, 2005.

\bibitem{Bon1} J.-M. Bony, Calcul symbolique et propagation des singularit\'{e}s pour les \'{e}quations aux d\'{e}riv\'{e}es partielles non lin\'{e}aires, Ann. Sci. \'{E}c. Norm. Sup. (4) \textbf{14} (1981) 209-246.










\bibitem{ClaEtAl} D. Clamond, D. Fructus, J. Grue and {\O}. Kristiansen, An efficient model for three-dimensional surface wave simulations. Part II: Generation and absorption, J. Comput. Phys. \textbf{205} (2005) 686-705.

\bibitem{Clem1} A. Cl\'{e}ment, Coupling of two absorbing boundary conditions for $2D$ time-domain simulations of free surface gravity waves, J. Comput. Phys. \textbf{126} (1996) 139-151.






\bibitem{CoSh1} D. Coutand and S. Shkoller, Well-posedness of the free-surface incompressible Euler equations with or without surface tension, J. Amer. Math. Soc. \textbf{20} (2007) 829-930.



\bibitem{CS} W. Craig and C. Sulem, Numerical simulation of gravity waves, J. Comput. Phys. \textbf{108} (1993) 73-83.





\bibitem{DDZ1}F. Dias, A. Dyachenko and V. Zakharov, Theory of weakly damped free-surface flows: A new formulation based on potential flow solutions, Phys. Lett. A \textbf{372} (2008) 1297-1302.

\bibitem{Duc1} G. Ducrozet, {\em Mod\'{e}lisation des Processus Non-Lin\'{e}aires de G\'{e}n\'{e}ration et de Propagation d'\'{E}tats de Mer par une Approche Spectrale,} Ph.D. thesis, Universit\'{e} de Nantes; Ecole Centrale de Nantes, 2007.







\bibitem{GMS2} P. Germain, N. Masmoudi and J. Shatah, Global existence for capillary water waves, Comm. Pure Appl. Math. \textbf{68} (2015) 625-687.







\bibitem{HIT1} J. Hunter, M. Ifrim and D. Tataru, Two dimensional water waves in holomorphic coordinates, Commun. Math. Phys. \textbf{346} (2016) 483-552.




\bibitem{IT3} M. Ifrim and D. Tataru, Two-dimensional gravity water waves with constant vorticity I: Cubic lifespan, Anal. PDE \textbf{12} (2019) 903-967.


\bibitem{Igu1} T. Iguchi, Well-posedness of the initial value problem for capillary-gravity waves, Funkcial. Ekvac. \textbf{44} (2001) 219-241.





\bibitem{IonPus5} A. Ionescu and F. Pusateri, Long-time existence for multi-dimensional periodic water waves, Geom. Funct. Anal. \textbf{29} (2019) 811-870.

\bibitem{IsrOrs1} M. Israeli and S. Orszag, Approximation of radiation boundary conditions, J. Comput. Phys. \textbf{41} (1981) 115-135.

\bibitem{JenEA1} G. Jennings, S. Karni and J. Rauch, Water wave propagation in unbounded domains. Part I: Nonreflecting boundaries, J. Comput. Phys. \textbf{276} (2014) 729-739.


\bibitem{KN} T. Kano and T. Nishida, Sur les ondes de surface de l'eau avec une justification math\'{e}matique des \'{e}quations des ondes en eau peu profonde, J. Math. Kyoto Univ. \textbf{19} (1979) 335-370.

\bibitem{Kato1} T. Kato, The Cauchy problem for quasi-linear symmetric hyperbolic systems, Arch. Rational Mech. Anal. \textbf{58} (1975) 181-205.

\bibitem{Kato2} T. Kato, Nonstationary flows of viscous and ideal fluids in $\R^3$, J. Funct. Anal. \textbf{9} (1972) 296-305.



\bibitem{Lan2} D. Lannes, Well-posedness of the water-wave equations, J. Amer. Math. Soc. \textbf{18} (2005) 605-654.

\bibitem{Lan1} D. Lannes, {\em The Water Waves Problem: Mathematical Analysis and Asymptotics,} Mathematical Surveys and Monographs 188 (Amer. Math. Soc., Rhode Island, 2013).

\bibitem{Lian1} J. Lian, Global well-posedness of the free-surface incompressible Euler equations with damping, J. Differential Equations \textbf{267} (2019) 1066-1094.








\bibitem{BMG1} B. M\'{e}sognon-Gireau, The Cauchy problem on large time for the water waves equations with large topography variations, Ann. I. H. Poincar\'{e} \textbf{34} (2017) 89-118.

\bibitem{Met1} G. M\'{e}tivier, {\em Para-differential Calculus and Applications to the Cauchy Problem for Nonlinear Systems,} CRM Series 5 (Edizioni della Normale, Pisa, 2008).







\bibitem{Nal} V.I. Nalimov, The Cauchy-Poisson problem, Din. Splo\v{s}n. Sredy \textbf{18} (1974) 104-210.

\bibitem{NgNi1}M. Ngom and D. Nicholls, Well-posedness and analyticity of solutions to a water wave problem with viscosity, J. Differential Equations \textbf{265} (2018) 5031-5065.






\bibitem{RT1} J. Rauch and M. Taylor, Exponential decay of solutions to hyperbolic equations in bounded domains, Indiana Univ. Math. J. \textbf{24} (1974) 79-86.















\bibitem{T2} M. Taylor, {\em Pseudodifferential Operators and Nonlinear PDE,} Progress in Mathematics 100 (Birkh\"{a}user, Boston, 1991).


\bibitem{T4} M. Taylor, Commutator estimates, Proc. Amer. Math. Soc. \textbf{131} (2002) 1501-1507.




\bibitem{Wu2} S. Wu, Well-posedness in Sobolev spaces of the full water wave problem in 2-D, Invent. Math. \textbf{130} (1997) 39-72.







\bibitem{Z} V. Zakharov, Stability of periodic waves of finite amplitude on the surface of a deep fluid, J. Appl. Mech. Tech. \textbf{9} (1968) 190-194.


\bibitem{Zhu1} H. Zhu, Control of three dimensional water waves, Arch. Rational Mech. Anal. \textbf{236} (2020) 893-966.

\end{thebibliography}
\end{document}